\documentclass{article}
\usepackage[affil-it]{authblk}
\usepackage{cite} 
\usepackage{tikz}
\usepackage{amsmath,amsthm, amsfonts}
\usepackage{color}
\usepackage[spanish,english]{babel}
\usepackage[latin1]{inputenc}
\usepackage{graphicx}
\usepackage{fancybox}
\usepackage{titlesec}
\usepackage{float}
\usepackage{enumerate}
\usepackage{multicol}
\usepackage{flushend}
\usepackage{pgfplots}
\usepackage[bookmarks]{hyperref}
\usepgfplotslibrary{polar}
\usepgflibrary{shapes.geometric}
\usetikzlibrary{calc}
\pgfplotsset{my style/.append style={axis x line=middle, axis y line=
middle, xlabel={$x$}, ylabel={$y$}, axis equal }}
\newtheorem{thm}{Theorem}[section]

\newtheorem{lem}[thm]{Lemma}
\newtheorem{prop}[thm]{Proposition}
\theoremstyle{definition}

\theoremstyle{remark}
\newtheorem{rem}[thm]{Remark}

\title{Global weak solutions for a $2\times 2$ balance nonsymmetric system of Keyfitz-Kranzer type }

\author[1]{Richard A. De la Cruz G.\footnote{{\it e-mail:} \href{mailto:richard.delacruz@uptc.edu.co}{richard.delacruz@uptc.edu.co}}}
  \author[2,3]{Juan C. Juajibioy \footnote{{\it e-mail:} \href{mailto:jcjuajibioyo@unal.edu.co}{jcjuajibioyo@unal.edu.co}}}
  \author[2]{Leonardo Rendon\footnote{{\it e-mail:} \href{mailto:lrendona@unal.edu.co}{lrendona@unal.edu.co}}}
  \affil[1]{Escuela de Matemáticas y Estadística, Universidad Pedagógica y Tecnológica de Colombia, Tunja}
  \affil[2]{Departamento de Matemáticas, Universidad Nacional de Colombia, Bogot\'a}
  \affil[3]{Departamento de Ciencias Naturales y Exáctas, Fundación Universidad Autonoma de Colombia, Bogot\'a}

\date{\today}
\begin{document}
\maketitle
\begin{abstract}\noindent
In this paper  we consider  the  existence  of global  weak entropy solutions
for a particualr nonsymmetric Keyfitz-Kranzer type system,  by using  the
compensated compactness method  we  get bounde entropy weak
solutions .
\end{abstract}
\section{introduction}
 In this chapter  we consider  the balanced nonsymmetric system 
\begin{equation}\label{intro1}
 \begin{cases}
  \rho_t+(\rho\phi(\rho,w))=f(\rho,w),\\
  (\rho w)_t+(\rho w\phi(\rho,w))_x=g(\rho,w)
 \end{cases}
 \end{equation}
where $\phi(\rho,w)=\Phi(w)-P(\rho)$, $\Phi$ a convex function. This
system was considered in \cite{lu3} where the author  showed the existence
of  global  weak solution  for the homogeneous system (\ref{intro1}). Another
system of  the type (\ref{intro1}) was considered in \cite{bt1} as a generalization to 
the  scalar Buckley-Leverett equations describing two phase  flow
in  poros media. The system (\ref{intro1}), recently, has been
object of constant  studies, in \cite{delta1} the author
considered  the particular  case in which  $\Phi(w)=w$, $P(\rho)=\frac{1}{\rho}$, in this case  the two charactheristics  of the 
system (\ref{intro1}) are linear degenerate, solving the
Riemann problem the existence and uniqueness of delta shock solution were established. In this line in \cite{nosyme1}
the  authors  considered the case $\Phi(w)=w$ and $P(\rho)=\frac{B}{\rho^{\alpha}}$
with $\alpha\in (0,1)$, the  existence and uniqueness of  solutions  to the the Riemann problem   was got by solving the  Generalize Rankine-Hugoniot condition.
In  both  cases, when $\Phi(s)=s$ the system (\ref{intro1})  models  vehicular traffic flow in a highway without
entry neither  exit of  cars, in this  case the source term  represents 
the entry or exit of  cars  see \cite{trasource},\cite{corli} and  reference therein
for more  detailed  description of source  term.\\[0.2in]
Noticing  that  when $w$ is  constant, the  system (\ref{intro1}) reduces
to the  scalar  balance  laws
\begin{equation}\label{escalar1}
 \rho_t+(\rho\Phi(w)-\rho P(\rho))_x=f(\rho,w),
\end{equation}
and from the  second  equation in (\ref{intro1}) $g$ should be of the  form
\[
 wg=f.
\]
Moreover, if  we make $h(\rho)=\rho\Phi(w)-P(\rho)$ 
the global  weak solution  of  the Cauchy problem
\begin{equation}\label{escalar2}
 \begin{cases}
 \rho_t+h(\rho)_x=f(\rho),\\
 \rho(x,0)=\rho_0(x)
\end{cases}
\end{equation}
there  exists if $h(\rho)$ is  a convex function and
the source  term is dissipative, i.e 
\begin{align}
 &h^{''}(\rho)=-(2P^{'}(\rho)+\rho P^{''}(\rho))>0\\
 &sf(s)\leq 0.
\end{align}
We assume the following  conditions,
\begin{enumerate}[$\text{C}_1$]
 \item $f$, $g$ are  Lipschitz functions   such that 
 \begin{equation}\label{s1}
 wg(\rho,w)=f(\rho,w), \ f(0,0)=0
 \end{equation}
\item There exist a constan $M>0$ such  that 
\begin{equation}\label{s2}
 sf(s)\leq 0,
\end{equation}
for $|s|>M$
\item The  function $P(\rho)$ satisfies
\begin{align}
&P(0)=0, \ \lim_{\rho \to 0} \rho P^{'}(\rho)=0, \ \lim_{\rho\to \infty}P(\rho)=\infty\\
&\rho P^{''}(\rho)+2P^{'}(\rho)<0,  \ \text{for} \ \rho>0               
\end{align}

\end{enumerate}
\begin{rem}
 By example if $f(\rho,m)=\rho$, then $g(\rho,m)=\rho w$ in 
 this case we have the nonsymmetric  system with lineal damping.
\end{rem}
Making $m=pw$, system (\ref{intro1}) can  be  transformed in a symmetric system
\begin{equation}\label{intro2}
 \begin{cases}
  \rho_t+(\rho\phi(\rho,m))=f(\rho,m),\\
  m_t+(\rho w\phi(\rho,m))_x=g(\rho,m)
 \end{cases}
 \end{equation}
for this  system. Making $F(\rho,m)=(\rho\phi(\rho,m),m\phi(\rho,m))$, then
\[
dF=
\begin{pmatrix} \phi+\rho\phi_{\rho}&\rho\phi_m\\
m\phi_{\rho}&\phi+m\phi_m \end{pmatrix},
\]
so the  eigenvalues and  eigenvector  of $dF$ are given  by

\begin{align}
 \lambda_1(\rho,m)& =\Phi(\frac{m}{\rho})-P(\rho)  & r_1&=(1,-\frac{\phi_{\rho}}{\phi_m}) \label{eq:1} \\
\lambda_2(\rho,m)& = \Phi(\frac{m}{\rho})-\rho P^{'}(\rho) &  r_2&=(1,\frac{m}{\rho}). \label{eq:2}
\end{align}
From (\ref{eq:1}), (\ref{eq:2}) the  k-Riemann  invariants are given  by
\begin{equation}\label{RI1}
 \begin{cases}
  W(\rho,m)=\Phi(\frac{m}{\rho})-P(\rho),\\
  Z(\rho,m)=\frac{m}{\rho}.
 \end{cases}
\end{equation}
Moreover
\begin{align}
 &\nabla \lambda_1\cdot r_1=0,\\
 &\nabla \lambda_2\cdot r_2=2P^{'}(\rho)+\rho P^{''}(\rho).
\end{align}
By  $\text{C}_3$ condition, the  system (\ref{aprio1}) is
linear  degenerate in the first charactheristic field, non linear degenerate
in the second  charactheristic field  and non strictly hyperbolic.
In this paper  we obtain  the main following theorem
\begin{thm}
 Let the  initial  data
 \begin{equation}\label{data1}
  \rho(x,0)=\rho_0(x), w(x,0)=w_0(x) \in L^{\infty}(\Omega),
 \end{equation}
whit  $\rho_0(x)\geq 0$. The total  variation of the Riemann invariants
$W_0(x)$ be  bounded, and the conditions $C_1$, $C_2$, $C_3$   holds, then the  Cauchy problem
(\ref{intro1}), (\ref{data1}) has  a  global bounded weak entropy solution and
$w_x(x,t)$ is  bounded in $L^1(\mathbb{R})$. Moreover
for $w$ costant  $\rho$ is the global weak solution
of  the  scalar  balance  laws
\begin{equation}\label{escalar3}
 \rho_t+h(\rho)_x=f(\rho),
\end{equation}
where $h(\rho)=\Phi(w)\rho-\rho P(\rho)$, $f(\rho)=f(\rho,w)$.
\end{thm}
\section{A priori bounds and  existence}
In order  to  get weak  solutions,  in this  section  we investigate the 
problem of the  existence of the solutions for the  parabolic  regularization
to the system (\ref{intro1})
\begin{equation}\label{aprio1}
 \begin{cases}
  \rho_t+(\rho\phi(\rho,w))=\epsilon \rho_{xx}+f(\rho,w),\\
  (\rho w)_t+(\rho w\phi(\rho,w))_x=\epsilon(\rho w)_{xx} +g(\rho,w)
 \end{cases}
 \end{equation}
with initial  data
\begin{equation}\label{aprio2}
 \rho^{\epsilon}(x,0)=\rho_0(x)+\epsilon, \ w^{\epsilon}(x,0)=w_0(x).
\end{equation}
We consider the transformation $m=\rho w$,  replacing in (\ref{aprio1}) we have
\begin{equation}\label{aprio3}
 \begin{cases}
  \rho_t+(\rho\phi(\rho,m))=\epsilon \rho_{xx}+f(\rho,m),\\
  m_t+(m\phi(\rho,m))_x=\epsilon m_{xx} +g(\rho,m),
 \end{cases}
 \end{equation}
with initial  data
\begin{equation}\label{datavanish}
 \rho(x,0)=\rho_0(x), \ \ m(x,0)=\frac{w_0(x)}{\rho_0(x)}.
\end{equation}

\begin{prop}
Let $\epsilon >0$ be, the Cauchy problem (\ref{aprio1}), (\ref{aprio2}),  has  a unique
solution for any $(\rho_0,w_0)$. Moreover if $(\rho_0,w_0)\in \Sigma$  thei
solutions $(\rho^{\epsilon},m^{\epsilon})$ satisfies
\begin{equation}\label{cotas1}
0<c\leq \rho(x,t)\leq M, \ \ |\frac{m(x,t)}{\rho(x,t)}|\leq M
\end{equation}
\end{prop}\noindent
The  proof of this  theorem  is postponed at the end of the  section. We begin
with some lemmas that will be usefull afterward.\\
Let $U=(\rho,m)^T$, $H(U)=(f(U), g(U))$ and  $M=DF$ where $F(\rho,m)=(\rho\phi,m\phi)$. Then the  system
(\ref{aprio3}) can be  written in the form
\begin{equation}\label{prop1}
 U_t=\epsilon U_{xx}+MU_x+H.
\end{equation}
For any $C_1$, $C_2$ constants let
\begin{align}
 G_1&=C_1-W,\label{prop:2}\\
 G_2&=Z-C_2\label{prop:3},
\end{align}
where $W$, $Z$  are the  Riemann invariants  given in (\ref{RI1}). We proof
that the  region 
\begin{equation}\label{prop3}
 \Sigma= \{(\rho,m): G_1\leq 0, G_2\leq 0\}
\end{equation}
is an invarian  region.
\begin{lem}
 If $\rho\in C^{1,2}([0,T]\times \mathbb{R})$ satisfies
 \begin{equation}\label{plower}
  \rho_t+(\rho\phi(w,\rho))_x=f(\rho,w),
 \end{equation}
with $\rho(0,\cdot)\geq 0$ and $w\in C^{1}(\Omega)$ then
$\rho(t,\cdot)\geq 0$, moreover if $\rho(0,\cdot)\geq \delta>0$
\begin{equation}\label{plower2}
 \int_0^T\int_{-\infty}^{\infty}\rho|w-w_0|dxdt<K
\end{equation}
with $k$ constant, then $\rho(x,t)\leq \delta(\epsilon,T)>0$
in $(0,T)$.

\end{lem}

\begin{lem}
The  function  $G_1$, $G_1$ defined in (\ref{prop:2}),(\ref{prop:3}),  are quasi-convex.
\end{lem}
\begin{proof}
Let $r=(X,Y)$ be a  vector.  If $r\cdot \nabla G_1=0$ then
$Y=X(\frac{m}{\rho}+\rho\frac{P^{'}(\rho)}{\Phi^{'}(w)})$ thus
\[
 \nabla^2G_1(r,r)=X^2\left(-\frac{1}{\rho}(2P^{'}(\rho+\rho P^{''}(\rho))+\Phi^{''}(w)(\frac{P^{'}(\rho)}{\Phi^{'}(w)})^2\right).
\]
If $r\cdot\nabla G_2=0$ then $Y=\frac{m}{\rho}X$, thus we have
\[
 \nabla^2G_2(r,r)=0.
\]
\end{proof}\noindent
\begin{lem}
If the condition $\text{C}_1$  holds then $G_1$, $G_2$ satisity
\begin{align}
\nabla G_1\cdot H\leq0,\\
\nabla G_2\cdot H\leq0,
\end{align}
\begin{proof}
From (\ref{prop:2})and (\ref{prop:3}),  we have  that $\nabla G_1=(-\Phi^{'}\frac{m}{\rho}-P{'},\Phi^{'}\frac{1}{\rho})$ and
$\nabla G_2=(-\Phi\frac{m}{\rho^2},\Phi^{'}\frac{1}{\rho})$ then
\begin{align}
 \nabla G_1\cdot H&=\frac{\Phi^{'}}{\rho}(-\frac{m}{\rho}f+g)+P^{'}f\leq 0,\\
  \nabla G_2\cdot H&=\frac{\Phi^{'}}{\rho}(-\frac{m}{\rho}f+g)\leq 0.
\end{align}
\end{proof}
\end{lem}
From the Theorem 1.3.1, the  region $\Sigma$  defined
in (\ref{prop3}) is  an invariant  region for the  system \ref{aprio3}. It
follows  from (\ref{prop:2}),(\ref{prop:3}) that
\begin{align}
 C_1&\leq \Phi(\frac{m}{\rho})-P(\rho),\notag \\
 \frac{m}{\rho}&\leq C_2.\notag
\end{align}
then
\begin{equation}\label{prop4}
 C_1-\Phi(C_2)\leq P(\rho),
\end{equation}
we appropiately choose  $C_1$, $C_2$ such that
\begin{equation}\label{prop5}
0<\delta\leq \rho, \ P(\rho)\leq \Phi(C_2)-C_1.
\end{equation}
By (\ref{prop4}) we have the proof of Proposition 4.1.1
\section{Weak convergence}
In this  section  we  show that the  sequence $(\rho^{\epsilon}, m^{\epsilon})$
has  a  subsequence that  converges the weak solutions  to the system (\ref{aprio3}). For  this  we
consider  the following  entropy-entropy  flux pairs construct  in \cite{lu3}   by the author
\begin{align}
 & \eta(\rho,m)=\rho F(\frac{m}{\rho}),\label{weak:1}\\
 & q(\rho,m)=\rho F(\frac{m}{\rho})\phi(\rho,m) \label{weak:2}
\end{align}
The Hessian matrix of $\eta$ is  given  by
\[\nabla^2\eta
 \begin{pmatrix}
  F^{''}\frac{m^2}{\rho^3} & -F^{''}\frac{m}{\rho}\\
  -F^{''}\frac{m}{\rho} & F^{''}\frac{1}{\rho}
 \end{pmatrix}
\]
then  we  have  that
\begin{equation}\label{weak3}
 \nabla^2\eta(X,X)=\frac{F^{''}}{\rho}(\frac{m}{\rho}\rho_x-m_x)^2,
\end{equation}
where $X=(\rho_x,m_x)$. If $(\eta,q)$ is  an entropy-entropy  flux pair,  multiplying  in (\ref{aprio3}) by  $\nabla\eta(\rho,m)$ we have
\begin{equation}\label{weak4}
 n_t+q_x=\epsilon\eta_{xx}-\epsilon \nabla^2\eta(X,X)+\nabla\eta \cdot G(\rho,m).
\end{equation}
Replacing the  equation (\ref{weak3}) in (\ref{weak4}) we have
\begin{equation}\label{weak5}
 n_t+q_x=\epsilon\eta_{xx}-\epsilon \frac{F^{''}}{\rho}(\frac{m}{\rho}\rho_x-m_x)^2. +\nabla\eta \cdot G(\rho,m).
\end{equation}
Chose  a  function $\varphi\in C_0^{\infty}(\mathbb{R}^2_+)$ satisying $\varphi=1$ on
$[-L,L]\times [0,T]$. multiplying (\ref{weak5}) by $\varphi$ and
integrate the  result in $(\mathbb{R}^2_+)$ we have
\begin{align}
&\int_{\mathbb{R}}(\eta\varphi)(x,T)-\int_{\mathbb{R}}(\eta\varphi)(x,0)-\int_0^T\int_{\mathbb{R}}(\eta\varphi_t+q\varphi_x)dxdt\ \notag\\
&=-\epsilon\int_0^T\int_{\mathbb{R}}\frac{F^{''}}{\rho}(\frac{m}{\rho}\rho_x-m_x)^2+\epsilon\int_0^T\int_{\mathbb{R}}\eta\varphi_{xx}dxdt-\int_0^T\int_{\mathbb{R}}\nabla\eta \cdot G(\rho,m)dxdt. \notag
\end{align}
From the Proposition 2.1  we have
\begin{equation}\label{weak6}
 \epsilon \int_0^T\int_{-L}^L\frac{F^{''}}{\rho}(\frac{m}{\rho}\rho_x-m_x)^2dxdt\leq C.
\end{equation}
As  a consequence of the inequality (\ref{weak6}) we have the following Lemma.
\begin{lem}
For  any $\epsilon>0$, if $(\rho,m)$ is  a solutions to the  Cauchy problem
(\ref{aprio3}), (\ref{datavanish}), then  $\sqrt{\epsilon}\rho_x,\sqrt{\epsilon}m_x$ are
bounded in $L_{loc}^2(\mathbb{R}^2_+)\subset \mathcal{M}_{loc}$.
\end{lem}\noindent
For any bounde set $\Omega\subset \mathbb{R}^2_+$ we have
\begin{equation}\label{weak7}
 \|\epsilon\eta_{xx}\|_{\text{W}^{-1,2}(\Omega)}=\sqrt{\epsilon}\|\eta\|_{\text{L}^{\infty}(\Omega)}\|\sqrt{\epsilon}u_x\|_{\text{L}^2(\Omega)}\sup_{\varphi}\|\varphi_x\|_{\text{L}^2}\to 0,\ \epsilon \to 0.
\end{equation}
and
\begin{equation}\label{weak8}
 \nabla\eta\cdot G(\rho,m)\in \text{L}^{\infty}(\Omega)\subset \text{L}^1(\Omega)\subset \mathcal{M}_{loc}.
\end{equation}
\begin{lem}
\begin{align}
& g(\rho)_t+\left(\int^{\rho}g^{'}(s)f^{'}(s)ds+g(\rho)\Phi(w)\right)_x,\label{weak:9}\\
& (\rho\Phi(w))_t+\left(\rho\Phi^2(w)+f(\rho)\Phi(w)\right)_x \label{weak:10}
\end{align}
are compact in $\text{H}^{-1}_{loc}(\mathbb{R}^2_+)$. Particularly
\begin{equation}
\rho_t+(\rho\Phi(w)-\rho P(\rho))_x 
\end{equation}
are  compact  in $\text{H}^{-1}_{loc}(\mathbb{R}^2_+)$.
\end{lem}
\begin{proof}
 The  proof of (\ref{weak:10}) is  a consequence  of the Lemma 3.1
 and  the inequalities (\ref{weak7}), (\ref{weak8}) and  the Murat's Lemma. Multiplying  in (\ref{weak:9}) by
 $g^{'}$ we have
 \begin{align}\label{weak11}
 & g(\rho)_t+\left(\int^{\rho}g^{'}(s)f^{'}(s)ds+g(\rho)\Phi(w)\right)_x, \notag\\
 &=g(\rho)_xx-\epsilon g^{''}(\rho)\rho^2_x+\rho g^{'}(\rho)\Phi(w)_x+g^{'}f(\rho,w).
 \end{align}
By a similar  argument in the inequality (\ref{weak7}) we have
\[
 \|\epsilon g(\rho)_xx\|_{\text{W}^{-1,2}}\to 0, \text{as} \ \epsilon \to 0.
\]
From  the Lemma 3.1  the last  term in (\ref{weak11}) is in $\mathcal{M}_{loc}$. By the Murat's Lemma
we conclud the proof of (\ref{weak:9}).
\end{proof}\noindent
Acording  to the Young'measures Theorem 1.2.1, there  exists a  probability measure $v_{x,t}$
associated with  the  bounded  sequence $(\rho^{\epsilon},w^{\epsilon})$ such that for
almost $(x,t)$, $v_{x,t}$ satisfies  the followin  Tartar  equation,
\begin{equation}\label{weak12}
\langle v_{x,t},\eta_1 q_2-\eta_2 q_1\rangle=\langle v_{x,t},\eta_1 q_2\rangle -\langle v_{x,t},\eta_2 q_1\rangle
\end{equation}
for any entropy-entropy flux pair. Here $\langle v_{x,t},f(\lambda)\rangle=\int_{\mathbb{R}^2}f(\lambda)dv_{x,t}(\lambda)$.
We consider the  following entropy-entropy  flux pairs
\begin{align}
 &\eta_1=\rho^{\epsilon} ,  &q_1=\rho^{\epsilon}(\Phi(w{\epsilon})-P(\rho^{\epsilon}))+w^{\epsilon},\\
 &\eta_2=\rho^{\epsilon}w^{\epsilon},  &q_2=\rho^{\epsilon}w^{\epsilon} (\Phi(w{\epsilon})-P(\rho^{\epsilon}))+(w^{\epsilon})^2.
\end{align}
Noticing that $\eta_1 q_2-\eta_2 q_1=0$ we have that
\[
 \overline{\eta_1 q_2}-\overline{\eta_2 q_1}=0,
\]
then
\[
\overline{\rho^{\epsilon}}\overline{\rho^{\epsilon}w^{\epsilon} (\Phi(w{\epsilon})-P(\rho^{\epsilon}))+(w^{\epsilon})^2} -\overline{\rho^{\epsilon}w^{\epsilon}}\overline{\rho^{\epsilon}(\Phi(w{\epsilon})-P(\rho^{\epsilon}))+w^{\epsilon}}
\]
%===========================================================
%-----------------------------------------------------------------
\bibliographystyle{amsplain}
\bibliography{state}
\end{document}